\documentclass[a4paper,12pt]{amsart}
\usepackage{ amssymb }
\usepackage{xcolor}
\usepackage{hyperref}
\usepackage{enumitem}

\usepackage[margin=2.7cm]{geometry}

\numberwithin{equation}{section} 

\newtheorem{lemma}{Lemma}[section]

\newtheorem{theorem}[lemma]{Theorem}

\newtheorem{corollary}[lemma]{Corollary}

\newtheorem{proposition}[lemma]{Proposition}

\theoremstyle{definition}
\newtheorem{definition}[lemma]{Definition}

\newtheorem{example}[lemma]{Example}

    \newcommand{\cH}{\mathcal{H}}
    \newcommand{\cD}{\mathcal{D}}
	\newcommand{\norm}[1]{\left\|#1\right\|}

	\newcommand{\g}{\mathfrak{g}}
	\newcommand{\h}{\mathfrak{h}}

    \newcommand{\tri}{\mathfrak{t}}
    \newcommand{\brt}{\bar t}
\newtheorem{lettertheorem}{Theorem}

\def\cF{{\mathcal{F}}}
\def\cL{{\mathcal{L}}}

\def\m{{\mathbf{m}}}

\def\t{{\mathbf{t}}}

\def\bk{{\mathbf{k}}}

\def\bp{{\mathbf{p}}}

\def\EXP{{\mathbb{E}}}

\def\Tor{{\mathbb{T}}}
\def\R{{\mathbb{R}}}
\def\Z{{\mathbb{Z}}}
\def\N{{\mathbb{N}}}
\def\k{{\textbf{k}}}
\def\brC{{\bar C}}

\def\hC{{\hat C}}

\def\DS{\displaystyle}

\newcommand{\ignore}[1]{}
\newcommand{\annotation}[1]{\leavevmode\raise.8ex\hbox to0pt{\hss\ensuremath{\curlyvee}\hss}\marginpar{\tiny\baselineskip6pt #1}}

\usepackage{cite}

\title[Random walks on nilmanifolds]{On Rapid mixing for random walks on  nilmanifolds.}
\author{DMITRY DOLGOPYAT}
\address{(Dmitry Dolgopyat) University of Maryland Department of Mathematics 4176 Campus Drive
College Park, MD 20742-4015 USA}
\email{dmitry@umd.edu}

\author{Spencer Durham}
\address{(Spencer Durham) University of Maryland Department of Mathematics 4176 Campus Drive
College Park, MD 20742-4015 USA}
\email{sdurham4@umd.edu}

\author{Minsung Kim}
\address{(Minsung Kim) Department of Mathematics,  Kungliga Tekniska Högskolan, Lindstedtsvägen 25, SE-100 44 Stockholm, Sweden}
\email{minsungdream@gmail.com}

\date{\today}

\begin{document}

\begin{abstract} 
We prove rapid mixing for almost all random walks generated by $m$ translations on an arbitrary nilmanifold under mild assumptions on the size of $m$. For several classical nilmanifolds, we show that $m=2$ suffices.
This provides a partial answer to the question raised in \cite{D02} 
about the prevalence of rapid mixing for random walks on homogeneous spaces.
\end{abstract}

\maketitle


\section{Introduction}

Let $G$ be a simply connected nilpotent Lie group and $\Gamma$ be a co-compact lattice so that  $M = G/\Gamma$ is a nilmanifold equipped with the Haar measure $\mu$.
{We shall also use $\mu$ to denote the Haar measure on $G$. It will not cause a confusion since the meaning will be always clear from the context.}
A translation on $M$ by $g\in G$ is the map  $x\Gamma\mapsto gx\Gamma$. In this paper, we study random walks by a finite set of translations on $M$. 
These random walks will be defined by an $m\in \N, \ x\in M$, a set 
$F:=\{g_1,\ldots, g_m\}\subset G$, and an associated probability vector $\vec p =  (p_1, p_2,\dots, p_m)$, i.e. $p_i>0$ and $\DS \sum p_i=1$. 
 The random walk is 
 a Markov chain where $x_n = g_k x_{n-1}$ with the probability $p_k$. Let $\langle F\rangle$ denote the semigroup generated by the set $F$.
Associated to the set $F$, we have an operator $\cL:C^r(M) \rightarrow C^r(M)$  defined by 
 $$\cL(A)(x) := \EXP(A(x_1)|x_0=x)  =\sum_{j=1}^m p_j A(g_j \cdot x),\ A \in C^r(M).$$  
 It follows that for any $N >1$,
{\[\DS (\cL^N A)(x) = \EXP_x(A(x_N))=\sum_{|W|=N} p_N(W)A(Wx),\] where $|W|$ is the length of the word $W$ and 
$\DS {p_N(W)=p_N(g_{w_N}\ldots g_{w_1})}=\prod_{i=1}^Np_{w_i}$ is the probability of having walked by $W$ at time $N$.} 

Given observables $A,B \in C^r(M)$, the correlation of $A$ and $B$ after time $N$ is given by
$$\overline\rho_{A,B}(N) =  \int (\cL^NA)(x)B(x)d\mu(x) - \int A(x) d\mu(x) \int B(x) d\mu(x).$$

\begin{definition}
A random walk is  \emph{rapid mixing} 
if given $q\in \N$, there are constants $C, r>0$  such that for any $A,B\in C^r(M)$ and $N\in\N$,  
$$|\overline\rho_{A,B}(N)| \leq C\norm{A}_{C^r(M)} \norm{B}_{C^r(M)} N^{-q}.$$
\end{definition}
{Below, in order to simplify the formulas, we only consider correlations of zero mean functions. This is sufficient since
every function can be decomposed as a sum of a zero mean function and a constant (the mean).
}

Mixing plays a key role in the study of statistical properties of dynamical systems.
Many classical systems are exponentially mixing (see e.g. \cite[Appendix A]{DFL}).
The random walks we consider do not exhibit exponential mixing,   
{so the best estimate we can hope for 
is that the mixing of $C^r$ observables occurs at a $O(N^{-q(r)})$ rate with $q(r)\to\infty$ as $r\to\infty$. 
This is exactly the definition of rapid mixing that we have provided. Our main result shows that most sufficiently rich walks are indeed rapid mixing.}
Rapid mixing is also sufficient to establish several
key statistical properties, {including the Central Limit Theorem, see Appendix \ref{AppCLT}}.

Nilmanifolds support a rich and well-studied variety of homogeneous dynamical systems. These systems, in addition to being of purely dynamical interest, are also of interest to the broader mathematical community as their dynamical properties can have important consequences in other fields, particularly 
{number theory \cite{Fu61, F17} and combinatorics \cite{HK18}}. 
In part due to the complexity of their algebraic structures,
dynamics on higher-dimensional nilmanifolds are not yet completely understood. In this paper, we prove the rapid mixing of almost all random walks on nilmanifolds generated by a {sufficiently large (finite) number of group elements.} In particular, we will define a technical algebraic condition called 
{\em $m$-greatness} and show that almost any random walk on an $m$-great nilmanifold supported on $m$ translations is rapid mixing. {We then show that every step-$s$ nilmanifold is $s$-great.} This leads to the following theorem which applies to all nilmanifolds.
\begin{lettertheorem}\label{allNilmanifoldsTheorem}
     For any step-$s$ nilmanifold $M=G/\Gamma$, there is a constant $N_G\leq s$ such that for $m\geq N_G$, almost every random walk generated by $m$ translations is rapid mixing.
\end{lettertheorem}

 Here and below we say that a property holds almost everywhere if it fails on a subset of zero $\mu^m$--measure.

We note that even in the case of the torus there are exceptional walks for which mixing can be slower than any power of time. In fact, \cite{D02} identifies a Diophantine condition which is {\em necessary} and sufficient for
rapid mixing. Therefore, {\em almost every} can not be replaced by every in Theorem \ref{allNilmanifoldsTheorem} (as well as Theorem \ref{ThRM} below). Additionally, we note that the rate of mixing we obtain cannot be improved to an exponential rate. This follows from the fact that modes on the torus of the type $e^{2\pi i \lambda x}$ mix at a rate that is exponential but the exponent tends to 1 along a subsequence of $\lambda$s going to infinity.

 The primary weakness of Theorem \ref{allNilmanifoldsTheorem} is that it sometimes requires more than two generators to guarantee that rapid mixing random walks are of full measure. Within certain special classes of nilmanifolds, we can overcome this deficiency. 
  Among these are the \emph{quasi-abelian} nilmanifolds, which have previously been studied in the context of parabolic flow dynamics \cite{Rav18,FF23}. We also give special attention to triangular nilmanifolds and step-3 nilmanifolds, and ultimately show the following theorem.
\begin{lettertheorem}
\label{ThRM}
 If $G$ is quasi-abelian, triangular, or step-3 or lower, then $N_G = 2$.
\end{lettertheorem}

It is a classical theorem that for single translations on a nilmanifold, ergodicity, unique ergodicity, and minimality are equivalent and in fact only depend on the irrationality of the abelian projection of the translation. Our result is similar in that we only need an arithmetic property on 
associated toral actions.
On the other hand, unlike other cases, the arithmetic property we require takes into account how higher-order commutators of the generators will act on $M$. In particular, the argument requires certain polynomials arising from the Lie algebra structure to be linearly independent. However, the relevant polynomials only depend on the projections of the generators in the abelianization of $G$.
It follows that Theorem \ref{allNilmanifoldsTheorem} could be strengthened to say that any $m$-tuples in $G$  that project into the aforementioned full measure set inside the maximal abelian factor generate rapid mixing walks. 

A rapid mixing random walk on $M$ also satisfies the Central Limit Theorem
(see Appendix \ref{AppCLT}). 
Several results on central and local limit theorems for a position of a walker on a nilpotent Lie groups were 
obtained recently
in \cite{hough2019local, BE23A, BE23B,DH21}.  In contrast, our result concerns 
additive functions of the walks on nilmanifolds.
\medskip

We now describe the structure of the paper. 
Section \ref{sec;prelim} contains preliminaries on nilpotent Lie groups, Lie algebras, and nilmanifolds. 
{In Section \ref{sec;proof}, we define $m$-greatness as well as a relevant Diophantine property. We show that the Diophantine property implies rapid mixing (Proposition \ref{pencilImpliesMixing}). Then, modulo a claim that connects $m$-greatness to the Diophantine property (Proposition \ref{lem;great}),  we show that a full measure set of $m$-tuples on an $m$-great nilmanifold generate rapidly mixing walks. In Section \ref{sec;RW}, we prove
Proposition \ref{lem;great} by constructing words that act Diophantinely on appropriate tori. In Section \ref{sec;greatGroups}, we establish that the groups listed in Theorem \ref{ThRM} are $2-$great implying Theorem \ref{ThRM}. We also show Proposition \ref{PrSStepSGreat} which states that any step $s$-nilmanifold is $s$-great, proving Theorem \ref{allNilmanifoldsTheorem}.} Finally, we provide an example of a Lie algebra that is not $2$-great, showing that our technique is not sufficient to show the expected optimal result that $N_G\!\!=\!\! 2$ for all nilpotent groups.

\subsection*{Acknowledgement} 
The authors {thank} Adam Kanigowski for fruitful discussions and several suggestions for improving the quality of the initial draft. They are grateful to Bassam Fayad and Giovanni Forni for their careful readings and several comments on the draft.
M.K. thanks Xuesen Na for his assistance in setting up a computer program for testing calculations.


D.D. was partially supported by the NSF grant DMS-2246983.

S.D. was partially supported by the NSF grant DMS-2101464.

M.K. was partially supported by UniCredit Bank R\&D group through the ‘Dynamics and Information Theory Institute’ at the Scuola Normale Superiore, foundations for the Royal Swedish Academy of Sciences, and Carl Trygger’s Foundation for Scientific Research.

\section{Preliminaries on nilmanifolds}\label{sec;prelim} 
We provide a background of nilpotent Lie groups, Lie algebras and nilmanifolds.
{ The material of this subsection is taken from \cite{CG90,Rag72}, and we also refer to \cite{AFRU21} for additional information about general nilmanifolds.}

\subsection{Nilpotent Lie groups and Lie algebras}
A real Lie algebra $\g $ is called \emph{nilpotent} if the lower (descending) central series of $\g$ terminates i.e. the sequence defined by 
\begin{equation}\label{def;descending}
\g = \g^{(0)} \supset \g^{(1)} = [\g,\g] \supset \cdots \supset \g^{(j)} = [\g^{(j-1)},\g] \supset \cdots,
\end{equation}
where $[\h ,\g] = \{[X,Y]:X \in \h, Y\in\g \}$, eventually has $\g^{(s)}=0$ for some $s$.
{ The \emph{step} of $\g$  is the minimal number $s$} that satisfies $\g^{(s)} = 0$.  

The lower central series of a Lie group $G$ is defined by $G^{(0)}=G$ and $G^{(j+1)}=[G^{(j)},G]$ where $[\cdot,\cdot] $ is the commutator bracket. A connected and simply connected Lie group $G$ is called nilpotent if $G^{(s)}$ is equal to the trivial group for some $s$. A Lie group $G$ is nilpotent if and only if its associated Lie algebra $\g$ is nilpotent. In fact, the lower central series of $G$ and $\g$ are connected, as $G^{(j)} =  \exp(\g^{(j)})$.
\begin{proposition}\cite{CG90}
If $G$ is a connected and simply connected nilpotent Lie group, then the exponential map $\exp : \g \rightarrow G$ is a diffeomorphism. 
\end{proposition}

\noindent
The product operation on $G$  satisfies the Baker-Campbell-Haussdorff (BCH) formula
\begin{equation}
\label{BCH}
\exp(X)\exp(Y) = \exp\left(X+Y + \frac{1}{2}[X,Y]+\sum_\alpha c_\alpha X_\alpha\right),
\end{equation}
where $\alpha$ is a finite (for nilpotent groups) set of labels, $c_\alpha$ are real constants, and $X_\alpha$ are iterated Lie brackets of $X$ and $Y$ (see \cite{Dyn47}).

It will be convenient to denote  $n_j =  \dim(\g^{(j)}) - \dim(\g^{(j+1)})$ so that $n_j$ is the dimension of the quotient algebra $\g^{(j)}/ \g^{(j+1)}$ (or corresponding quotient group $G^{(j)} / G^{(j+1)}$). If $X,Y\in\g$ satisfy $X-Y\in\g^{(j)}$ we write that $X=Y \mod \g^{(j)}$.

\begin{definition}[Malcev basis]
A Malcev basis for $\g$ through the descending central series $\g^{(j)}$ 
 is a basis $X_1^{(0)},\cdots X_{n_0}^{(0)}, $ $\cdots, X_1^{(s-1 )},\cdots, X_{n_s}^{(s-1)}$ of $\g$ satisfying the following:

(1) if we set $E^j = \{ X_1^{(j)},\cdots, X_{n_j}^{(j)} \}$, the elements of the set  $E^j \cup E^{j+1}\cup \cdots \cup E^{s}$ form a basis of $\g^{(j)}$;
 
(2) if we drop the first $l$ elements, the remaining elements span an ideal (of codimension $l$) of $\g$.
 \end{definition}
If $\Gamma=\DS  \left\{\exp\left(\sum_{j,k} m_{j,k} X_j^{(k)}\right)\right\}_{m_{j,k}\in \Z}$, then
we say that the basis is \emph{strongly based at} $\Gamma$. The Lie algebra $\g$ of any nilmanifold $G/\Gamma$ can be equipped with a Malcev basis strongly based at $\Gamma$. Moving forward, we will always use this basis when writing Lie algebra elements in coordinates. For convenience, we will denote by $X_i = X_i^{(0)}$. { We observe that if $A\subseteq\R^{n_0}$, then 
\[\mu(\{\alpha_{01}X_1+\ldots \alpha_{0n_0}X_{n_0}+\dots:
(\alpha_{01},\ldots,\alpha_{0n_0})\in A\})={ \mathrm{Leb}(A)},\] 
i.e. for a set defined by the coefficients of the $X_i$, the measure of the set is exactly the Lebesgue measure of the permitted coefficients.}

\subsection{Nilmanifolds and Fibration}
{A compact \emph{nilmanifold} is a quotient $M := G/\Gamma$} where $G$ is a nilpotent group and $\Gamma$ is a (co-compact) lattice of $G$.
The lattice $\Gamma$ exists if and only if $G$ admits rational structural constants.  We will consider the left action of $G$ by translations on $M$. More precisely, for $g,h \in G$, set 
$
g(h \Gamma)=(gh)\Gamma.
$
Every nilmanifold is a fiber bundle over a torus. The abelianization  $G^{ab} = G / [G,G]$ is abelian, connected and simply connected, hence isomorphic to $\R^n$. 
Thus, there is a natural projection 
\begin{equation}
P: G/ \Gamma \rightarrow  G^{ab}/\Gamma^{ab} \simeq \Tor^n. 
\end{equation}
For all $k \in \N_0$, the group $G^{(k+1)}$ is a closed normal subgroup of $G$, we have natural epimorphisms $\pi^{(k)}: G \rightarrow G/G^{(k+1)}$. Then, the group 
 $G^{(k+1)} \cap \Gamma$ is a lattice of $G^{(k+1)}$. Moreover, 
 $\Gamma_k:=\pi^{(k)}(\Gamma)$ is a lattice in $N_k:=G/G^{(k+1)}$ and 
 $$M^{(k)}:= G/G^{(k+1)}\Gamma=N_k/\Gamma_k$$ is a nilmanifold.
It follows that 
$\DS \pi^{(k)}: M = G/\Gamma \rightarrow M^{(k)}$
is a fibration
whose fibers are the orbits of $G^{(k+1)}$ on $G/\Gamma$, homeomorphic to the nilmanifolds $G^{(k+1)}/(G^{(k+1)} \cap \Gamma)$.  
We can also define $M_p:=G^{(p)}/G^{(p+1)}\Gamma_p\simeq \Tor^{n_p}$.

\begin{subsection}{Harmonic Analysis on Nilmanifolds}
In this section, we discuss the structure of $L^2(M)$ for a nilmanifold $M$. This discussion is based on Kirillov Theory. A general overview of the theory can be found in \cite{CG90}, and \cite{FF23} includes a similar discussion to ours focused on the theory of filiform nilmanifolds. Let $\mathfrak{a}^*$ denote the linear dual to a Lie algebra $\mathfrak{a}\subset\g$, and let $\mathfrak{a}^0$ denote the annihilator of $\mathfrak{a}$ in $\g^{*}$. For $\lambda\in\g^*$, let $\chi_\lambda:G^{(p)}\rightarrow\mathbb{C}$ be the function \[\chi_\lambda (g)=e^{2\pi i\lambda(\log(g))}.\] 
By Kirillov Theory, for every infinite dimensional irreducible unitary representation $H$ of $G$, there is an associated $0\leq p<s$ and $\lambda\in \g^{(p)*}\cap(\g^{(p+1)})^{0}$, such that for $\varphi\in {H}$, \[g\varphi=\chi_\lambda(g)\varphi\text{ for all }g\in G^{(p)}.\]  Observe that the functionals that appear in this description, $\g^{(p)*}\cap(\g^{(p+1)})^{0}$, are naturally associated with the functionals in the space $(\g^{(p)}/\g^{(p+1)})^*\simeq (\R^{n_p})^*$.

We need to describe which representations appear in $L^2(M)$. Since we have a basis for $\g$ (the Malcev basis), the coordinate functionals $X_1^{(1)*},\ldots,X_{n_{s-1}}^{(s-1)*}$ are well-defined and form a basis for $\g^*$. Define $\pi_p:\g\rightarrow\R^{n_p}$ by $\pi_p(V)=(X_1^{(p)*}(V),\ldots,X_{n_p}^{(p)*}(V))$. Also, let $E_p^*=\{X_1^{(p)*},\ldots,X_{n_p}^{(p)*}\}$, and define $\Lambda_p$ to be the set of integer linear combinations of elements of $E_p^*$. One may note that there is a natural correspondence between the $\Lambda_p$ and the characters on the torus $M_p$. The significance of $\Lambda_p$ is that the $\lambda\in\Lambda_p$ are the only functionals in $\g^{(p)*}\cap(\g^{(p+1)})^{0}$ with the property that $\chi_\lambda$ descends to a well defined function on $M$. Thus, for $\lambda\in\Lambda_p$ we can define \[H_\lambda=\{\varphi\in L^2(M):\varphi(gx)=e^{2\pi i\lambda(\log (g))} \varphi(x) \text{ for all } g\in G^{(p)}\},\] and $L^2(M)$ will decompose into a direct sum of such $H_\lambda$.
In particular, if we let $\displaystyle\Lambda=\Lambda_0\cup\bigcup_{p=1}^{s-1}(\Lambda_p-\{0\})$, 
we have that \[L^2(M)=\bigoplus_{\lambda\in\Lambda} H_\lambda.\]
Moreover, when $\varphi\in C^r(G/\Gamma)$, if we write
$\displaystyle \varphi=\sum_{\lambda\in\Lambda} \varphi_\lambda$ with $\displaystyle  \varphi_\lambda\in H_\lambda,$
then the weight functions $\varphi_\lambda$  satisfy 
$\displaystyle \|\varphi_\lambda\|_{C^0}\leq \frac{\|\varphi\|_{C^r}}{\|\lambda\|^r}.$
\end{subsection}
\subsection{Main examples}\label{sec;fil}
In this article, 
we will pay special attention to two classes of Lie groups. 
{The first class was first in \cite{Ver70}, 
(see also \cite{F17} for general introduction). }

\begin{definition}\label{def;quasi-abelian}
A 
Lie algebra $\mathfrak{g}$ 
is called \emph{quasi-abelian} if it is not abelian and has an abelian subalgebra of codimension 1.
\end{definition}

\noindent
Any quasi-ableian Lie algebra $\g$ has a basis
$(X, Y_{i,j})_{(i,j)\in J}$, satisfying the commutation relations
\begin{equation*}
[X, Y_{i,j} ] = Y_{i+1,j},\ (i, j) \in J,
\end{equation*}
and all other commutation relations are trivial (see \cite{FF14}). We also remark that the class of quasi-abelian Lie algebra contains the class of filiform Lie algebras
{ (see \cite{Ver70})}, so our results 
also hold in the filiform case. 
\medskip

The next example is not quasi-abelian, but it has a tractable structure.
\begin{definition}
{ A Lie algebra is {\em triangular} if it isomorphic to $\tri_s$--the Lie algebra of strictly upper triangular $(s+1)\!\times\! (s+1)$ matrices with the standard bracket
for some $s$.}
\end{definition}
It follows from the definition that $\tri_s$ is $\frac{1}{2}s(s+1)$-dimensional. Letting $E_{ij}$ represent the $(s+1)\times (s+1)$ matrix with a 1 in position $(i,j)$ and zeroes elsewhere, we see that $\{E_{ij}:j>i\}$ forms a basis for $\tri_s$. The relationship among these basis elements are given by \[[E_{ij},E_{i'j'}]=\delta_{ji'}E_{ij'}-\delta_{j'i}E_{i'j}.\] 
We refer to \cite{CG90, K22} for additional information on triangular algebras. 

\begin{subsection}{Diophantine Conditions}
Given a vector $v\in\R^d$, we say that $v\in {DC(\gamma,\tau)}$ if for all $n\in \Z^d$ and $m\in\Z$ \[|n\cdot v-m|\geq \frac{\gamma}{|n|^\tau},\]
Let $\displaystyle DC=\bigcup_{\gamma,\tau>0}DC(\gamma,\tau).$

In \cite{KM98}, Kleinbock and Margulis show that within submanifolds of $\R^n$ satisfying a linear independence condition almost every point is in $DC$. 
In particular, they show the following which is listed as Conjecture $H_1$ in their paper.
\begin{theorem}\label{thm;KM}\cite{KM98}
Let $f_1,\ldots,f_n$ along with the constant function 1 form a linearly independent set of analytic functions from $\R^m$ to $\R$. For almost every $x\in \R^m$, 
$(f_1(x),\ldots,f_n(x))\in DC$.

\end{theorem}
 This result will provide the bridge from the Lie algebra structure to the Diophantine estimates needed to show mixing. {Since the functions we will consider will always be homogeneous polynomials with positive degree, linear independence of $f_1,\ldots,f_n$ is the only condition we will need to check in order to apply this theorem in practice.}
\end{subsection}

\section{ $m$-greatness and mixing}\label{sec;proof} 
In this section, we will define  $m$-great nilmanifolds and show that on these manifolds,
almost any $m$-tuple generates a rapidly mixing walk.

\subsection{Linear independence of polynomials} \label{sec;pencils}
Let $\cH(t, \alpha)=(\cH_1(t, \alpha), \dots, \cH_\ell(t, \alpha))$ be a polynomial map of $\R^a\times \R^b\to \R^\ell.$ Let 
$$\cD(\cH)=\{\brt\in \R^a: \cH_1(\brt, \alpha), \dots ,\cH_\ell(\brt, \alpha)\text{ are linearly dependent}\}.$$
Expanding 
$\DS \cH_j(\brt, \alpha)=\sum_{\m} c_{j, \m}(\brt) \alpha^{\m}$ we see that the above conditions amounts
to vanishing of certain minors of the matrix $c_{j,\m}$ whence $\cD(\cH)$ is an algebraic subvariety of $\R^a.$
We say that $\cH$ is {\em non-degenerate} if $\cD(\cH)\neq \R^a.$

Now we describe a special polynomial mapping associated to a Lie algebra. Fix integers $m$ and $p.$ 
Consider $m$ vectors $\DS V_i=\sum_{j=1}^n \alpha_{ij} X_j$ with $i\in\{1, \ldots, m\}$. {These will correspond to the exponential coordinates of $m$ generators of a random walk $g_1,\ldots,g_m$ mod $\g^{(1)}$. Thus, each variable $\alpha_{ij}$ corresponds to the $X_j$ coordinate of $g_i$. The idea is to see in coordinates what group elements will arise as we take iterated brackets.} To that effect, let 
$$ \cH_{m,p} (\bk, \alpha)=[\ldots \bf{[k_0V ,k_1V],\ldots],k_p V]} +\g^{(p+1)}\in \g^{(p)} / \g^{(p+1)},$$
where $\bk = (\bk_0,\ldots,\bk_p)$, $\bk_q=(k_{q1}, \dots, k_{qm})$ and $\DS \bk_q \textbf{V}=\sum_{i=1}^m k_{qi} V_i.$ 
The image of $\cH_{m,p}$ is exactly the $\g^{(p)}$ coordinates of the elements $p$-fold brackets of vectors in the span of the $V_i$. As we shall see, this will bear relation to the elements $G$ achievable as $p$-fold brackets of elements of the group generated by the $g_i$. {We use $\bk$ instead of $t$ for the first parameter of this map because, while it still makes sense to evaluate this map at any real {value}, we will be particularly interested in $\mathcal{H}_{m,p}(\bk,\alpha)$ that occurs as differences between two different words in the walk. These differences correspond to special values of $\bk$, where $\bk$ counts how many times each generator was used at each step of the construction of the words used to obtain the difference $\mathcal{H}_{m,p}(\bk,\alpha)$, and thus will always be integral.}
Note that we have 
\begin{equation}\label{pencilDefinition}
\cH_{m,p} (\bk, \alpha)=\sum_{i_j\in\{1,\ldots,m\}} k_{0i_0}\ldots k_{pi_p}M_{i_0 i_1,\ldots,i_p}+\g^{(p+1)}
\end{equation} 
where $M_{i_0 i_1,\ldots,i_p}=[\ldots [V_{i_0} ,V_{i_1}],\ldots],V_{i_p}].$ In coordinates, we see that \[\cH_{m,p}(\bk,\alpha)=\sum_{\ell=1}^{n_p}P_\ell(\bk,\alpha)X_i^{(p)}+\g^{(p+1)}\] where each $P_\ell(\bk,\alpha)$ is a polynomial in the coordinates of $\bk$ and $\alpha_{ij}$ (in particular the $\alpha_{0j}$). {We will care about when the map $\cH_{m,p}$ is degenerate, 
since $\cD(\cH_{m,p})$ represents the variety \[\{{\bf k}\in \R^{m(p+1)}: P_1({\bf k}, \alpha), \dots, P_{n_p}({\bf k}, \alpha)\text{ are linearly dependent polynomials in }\alpha\}.\] } From {now on} we will suppress the $+\g^{(p+1)}$ that ought to appear whenever we discuss $\cH_{2,p}$, but we {agree} that $\cH_{2,p}$ is always defined modulo $\g^{(p+1)}$. 
We now provide a brief illustrative example. 
\begin{example}\label{bracketExample}
     Let $\g$ be the step-3 Lie algebra of dimension 5 with the following commutation relations 
     \begin{align*}[X_1,X_2]=Y, \quad
[Y,X_1]=Z_1, \quad 
[Y,X_2]=Z_2.\end{align*}
with all other brackets being $0.$
Let 
$V_1=\alpha_{11} X_1+\alpha_{12} X_2+\dots$, $V_2=\alpha_{21} X_1+\alpha_{22} X_2+\dots$.
Then
$\DS [[V_1,V_2],V_1]=(\alpha_{11}^2\alpha_{22}-\alpha_{11}\alpha_{12}\alpha_{21})Z_1+(\alpha_{11}\alpha_{12}\alpha_{21}-\alpha_{12}^2\alpha_{21})Z_2 \mod \g_3.$ 
Therefore
$\DS M_{121}=\begin{bmatrix}\alpha_{11}^2\alpha_{22}-\alpha_{11}\alpha_{12}\alpha_{21}\\\alpha_{11}\alpha_{12}\alpha_{21}- \alpha_{12}^2\alpha_{21} \end{bmatrix}.$

    Suppose now we wish to compute the $\cH_{ 2,2}$ for this $\g$. To do so, we can proceed in two different ways. We could compute each $M_{i_1i_2i_3}$ and then write the sum as in \eqref{pencilDefinition}. Alternatively, we can explicitly compute $[\bf{[k_0V,k_1V],k_2V]}$ and we shall see how the $M_{i_1i_2i_3}$ appear. We will take this second approach. 
    \begin{align*}
    [[k_{01}V_1+k_{02}V_2&,k_{11}V_1+k_{12}V_2],k_{21}V_1+k_{22}V_2]\\
    =[(k_{01}k_{12}&(\alpha_{11}\alpha_{22}-\alpha_{12}\alpha_{21})-k_{02}k_{11}(\alpha_{11}\alpha_{22}-\alpha_{12}\alpha_{21}))Y,k_{21}V_1+k_{22}V_2]\\
    &=k_{01}k_{12}k_{21}\big((\alpha_{11}^2\alpha_{22}-\alpha_{11}\alpha_{12}\alpha_{21})Z_1+(\alpha_{11}\alpha_{12}\alpha_{22}-\alpha_{12}^2\alpha_{21} )Z_2\big) \\
    &+k_{01}k_{12}k_{22}\big((\alpha_{11}\alpha_{21}\alpha_{22}-\alpha_{12}\alpha_{21}^2)Z_1+(\alpha_{11}\alpha_{22}^2-\alpha_{12}\alpha_{21}\alpha_{22})Z_2 \big) \\
    &-k_{02}k_{11}k_{21}\big((\alpha_{11}^2\alpha_{22}-\alpha_{11}\alpha_{12}\alpha_{21})Z_1+(\alpha_{11}\alpha_{12}\alpha_{22}-\alpha_{12}^2\alpha_{21} )Z_2\big) \\
    &-k_{02}k_{11}k_{22}\big((\alpha_{11}\alpha_{21}\alpha_{22}-\alpha_{12}\alpha_{21}^2)Z_1+(\alpha_{11}\alpha_{22}^2-\alpha_{12}\alpha_{21}\alpha_{22})Z_2 \big).
    \end{align*}

We see that $M_{i_1i_2i_3}$ is exactly the coefficient of $k_{0i_1}k_{1i_2}k_{2i_3}$, and our mapping can 
be written as \[k_{01}k_{12}k_{21}M_{121}+k_{01}k_{12}k_{22}M_{122}+k_{02}k_{11}k_{21}M_{211}+k_{02}k_{11}k_{22}M_{212}.\] In {vector form, $\cH_{2,2}$ takes form}
    \begin{align*}&k_{01}k_{12}k_{21}\begin{bmatrix}\alpha_{11}^2\alpha_{22}-\alpha_{11}\alpha_{12}\alpha_{21}\\\alpha_{11}\alpha_{12}\alpha_{21}- \alpha_{12}^2\alpha_{21} \end{bmatrix}
    +k_{01}k_{12}k_{22}\begin{bmatrix}\alpha_{11}\alpha_{21}\alpha_{22}-\alpha_{12}\alpha_{21}^2\\\alpha_{11}\alpha_{22}^2- \alpha_{12}\alpha_{21}\alpha_{22} \end{bmatrix}\\
    &-
    k_{02}k_{11}k_{21}
    \begin{bmatrix}\alpha_{11}^2\alpha_{22}-\alpha_{11}\alpha_{12}\alpha_{21}\\\alpha_{11}\alpha_{12}\alpha_{21}- \alpha_{12}^2\alpha_{21} \end{bmatrix}-k_{02}k_{11}k_{22}\begin{bmatrix}\alpha_{11}\alpha_{21}\alpha_{22}-\alpha_{12}\alpha_{21}^2\\\alpha_{11}\alpha_{22}^2- \alpha_{12}\alpha_{21}\alpha_{22} \end{bmatrix}.\end{align*} 
    We clearly see that $\cH_{2,2}$ is non-degenerate
    since even just setting $k_{01}=k_{12}=k_{21}=1$ and $k_{02}=k_{11}=k_{22}=0$, we get that at that value of $\bk$, 
    $\cH_{2,2}$ evaluates to $M_{121}$, which has linearly independent rows. {Thus, our selected value of ${\bf k}$ is not in $\cD(\cH_{2,2})$, and thus $\cH_{2,2}$ is non-degenerate.}
\end{example}

\subsection{$m$-great Lie algebras and mixing}
Now, we will define the key technical property which guarantees that the set of rapid mixing $m$-tuples is of full measure. \begin{definition}\label{def;great}
A Lie algebra $\g$ is \emph{$m$-great} if $\cH_{m,i}$ is non-degenerate for every $i\in[0,s-1]$.
\end{definition} 
 When a Lie algebra is not $m$--great, we say it is $m$--bad. We will call a Lie group or a nilmanifold $m$-great if its associated Lie algebra is $m$-great. We will show that $m$-greatness implies that the walk generated almost every $m$-tuple has a Diophantine property in the following sense.
 
 \begin{definition}\label{niceDef}
     Let $M=G/\Gamma$ be a step-$s$ nilmanifold and $F=(g_1,\ldots, g_m)$ be an $m$-tuple.
     \begin{itemize}
         \item A pair of words $W_1,W_2\in\langle F\rangle$ is said to be \textit{nice} on $G^{(p)}$ if they have equal length and their difference $h:=W_2W_1^{-1}$ satisfies $h\in G^{(p)}$, and 
         $\pi_p(\log(h))\in DC$.
         \item An $m$-tuple is said to act \textit{Diophantinely} on $M_p$ if there exists a pair of words that is nice on $G^{(p)}$.
         \item We say that an $m$-tuple acts \textit{Diophantinely at all levels} if it acts Diophantinely on $M_p$ for all $0\leq p<s.$   
     \end{itemize}  
 \end{definition}
 Let $E_m$ denote the set of $m$-tuples which act Diophantinely at all levels on a nilmanifold $M$, {and recall that $\mu^m$ is the measure on the space of $m$-tuples of elements of $G$ given by the $m$-fold product of the Haar measure on $G$ with itself.}

\begin{proposition}\label{lem;great}
If $M\!\!=\!\!G/\Gamma$ is $m$-great, then $E_m$ is of full measure with respect to $\mu^m$.
\end{proposition}
We postpone the proof to Section \ref{sec;RW}. We now show that the random walk generated by any $m$-tuple that act Diophantinely on all levels is rapid mixing. The key estimate appears in the following theorem:

\begin{proposition}
\label{ThTame}
If a random walk is generated by an $m$-tuple acting Diophantinely at all levels, then there exist constants $C, a, b>0$ such that for $A \in C^r(M)$ with zero mean,
\begin{equation}
\label{Tame}
\|\cL^N A\|_{C^0}\leq \frac{K \|A\|_{C^r}}{N^{ar-b}}.
\end{equation}
\end{proposition}

Whenever (\ref{Tame}) holds, it implies rapid mixing of the associated random walk since
\begin{equation}|\overline\rho_{A,B}(N)| \leq \|\cL^NA\|_{C^0}\|B\|_{C^0}\leq C\norm{A}_{C^r} \norm{B}_{C^0} N^{-(ar-b)}.\label{spectralGapToRM}\end{equation}

To prove Proposition \ref{ThTame}, 
we will show that for any $\lambda\in\Lambda$, $\cL$ has spectral gap on $H_\lambda$, and, moreover, that the size of the spectral gap has a polynomial lower bound {in the norm of} $\lambda$. We will also use the fact that the size of a $C^r$ function's projection into $H_\lambda$ is polynomially bounded in the norm of $\lambda$. Together, these estimates will allow us to prove Theorem \ref{ThTame}. We now show the spectral gap.

\begin{lemma}
\label{HChiSp}
Assume that $F$ acts Diophantinely at all levels. There exist constants $C_1,C_2,\ell,\tau>0$ such that if $\lambda\in\Lambda$ is a nontrivial functional and $A\in H_\lambda$, then  
$$\DS \|\cL^N A \|_{C^0}\leq C_1 \left(1-C_2 \norm{\lambda}^{ -2 \tau}\right)^{N/\ell}
\|A\|_{C^0} .$$
\end{lemma}

\begin{proof}
{Let $A\in H_\lambda$ where $\lambda\in \Lambda_i$. We will now fix words and constants that depend only on $i$ and not on the particular $\lambda$ as follows. Let $W^i_1$ and $W^i_2$ be the nice words on $G^{(i)}$ guaranteed by the assumption that the $g_i$ act Diophantinely at all levels. Let $\ell_i$ be the common length of $W^i_1$ and $W^i_2$, and let $\gamma_i,\tau_i$ be the associated Diophantine constants. 

Recall that $p_{\ell_i}(W)$ denotes the probability that the walk has moved by $W$ at time $\ell_i$. Suppressing the dependence on $\ell_i$, let $\bp=\min(p(W_1), p(W_2))$.
Suppose without loss of generality that 
$\bp=p(W_1)$. 
Then set $h=W_2W_1^{-1}$, and let $W_3,\ldots, W_{m^{\ell_i}}$ denote the remaining words with positive probability at step $\ell_i$ of the walk. With this notation, we have} 
\[\cL^{\ell_i}(A)(x)=\bp \left[A(W_1 x) +A(h W_1 x)\right]+ \left[(p(W_2)-p(W_1)) A(W_2 x)+\sum_{j\geq 3} p(W_j) A(W_j x) \right].\]
{ Since $A\in H_\lambda$, we have that \[\bp \left[A(W_1 x) +A(h W_1 x)\right]=\bp\left[A(W_1x)+\chi_\lambda(h)A(W_1x)\right]=(1+\chi_\lambda(h))A(W_1x).\] We then get that \[|1+\chi_\lambda(h)|=|1+e^{2\pi i\lambda(\log(h))}|.\] By niceness of $W_1^i$ and $W_2^i$, we have that $\log( h)\in DC(\gamma_i,\tau_i)$. Since $\lambda$ corresponds to an integral functional,
Lemma \ref{rem:newer} from Appendix A says that for some constant $\bar C_i$ depending only on $\gamma_i$, we have that
\[|1+\chi_\lambda(h)|\leq 2-\bar C_i\|\lambda\|^{-2\tau_i}.\]
Putting this all together, we find that
$$\bp|A(W_1 x) +A(h W_1 x)|=\bp|1+\chi_\lambda(h)| |A(W_1 x)|  \leq 2 \bp  (1-\brC_i \|\lambda \|^{-2 \tau_i})\|A\|_{C^0}. $$ Turning our attention to the second term in the expression for $\cL^{\ell_i}(A)(x),$ we see that \[(p(W_2)-p(W_1)) A(W_2 x)+\sum_{j\geq 3} p(W_j) A(W_j x) \leq (1-2\bp)\|A\|_{C^0}.\]}
Combining these estimates, we see that on $H_\lambda$,
\[\DS \|\cL^{\ell_i} A \|_{C^0} \leq  2 \bp  (1-\brC_i \|\lambda \|^{-2 \tau_i})\|A\|_{C^0}+(1-2\bp)\|A\|_{C^0}\leq(1-\hC_i \|\lambda\|^{-2 \tau_i})\|A\|_{C^0} ,\] where $\hat C_i=2\bp\bar C_i$.
    {We have now shown the desired result except for {an unwanted dependence on the level $i$ in the constants $\ell_i, \hat C_i$, and $\tau_i$}. Since there are only $s$ different levels, we only used $s$ different word pairs with associated constants. Thus, we can simply fix $\hC=\min\{\hC_i\}$, $\ell=\max\{\ell_i\}$, and $\tau=\max\{\tau_i\}$.} We then see that \[\DS \|\cL^\ell A \|_{C^0}\!\! \leq \!\! (1\!-\!\hC \|\lambda\|^{-2 \tau})\|A\|_{C^0}\] holds independently of the value of $i$. Iterating this estimate $\lceil N/\ell \rceil$ times gives the result.
\end{proof}
{Now we put together the estimates on the behavior of $\cL$ on each $H_\lambda$ 
and estimate how $\cL$ acts on  $C^r.$}

\begin{proof}[Proof of Proposition \ref{ThTame}]
Decompose $A$ as $\displaystyle\sum_{\lambda\in \Lambda} A_\lambda$ where $A_\lambda\in H_\lambda.$ Since $A$ has mean 0, $A_\lambda=0$ when $\lambda=0$. Thus,
\[\DS \|\cL^N A\|_{C^0}\leq 
\sum_{
\|\lambda\|^{2\tau}<\sqrt{N}}\|\cL^N A_\chi\|_{C^0}+
\sum_{
\|\lambda\|^{2 \tau}\geq \sqrt{N}}\|\cL^N A_\lambda\|_{C^0}.\]
For the first term, we will use the spectral gap of $\cL$ from Lemma \ref{HChiSp}. This gives us that when 
$\|\lambda\|^{2\tau}<\sqrt{N}$, \[\|\cL^N A_{\lambda}\|_{C^0}\leq 
C_1\left(1-\frac{C_2}{\sqrt{N}}\right)^{N/\ell}\|A\|_{C^0}.\]
Thus, since there are at most $N^{{\dim(M)/4\tau}}$ functionals in $\Lambda$ satisfying $\|\lambda\|^{2\tau}<\sqrt{N}$,
\[\sum_{
\|\lambda\|^{2\tau}<\sqrt{N}}\|\cL^N A_\lambda\|_{C^0}\leq C_1{ N^{\dim(M)/4\tau}}
\left(1-\frac{C_2}{\sqrt{N}}\right)^{N/\ell}\|A\|_{C^0}.\]
Setting $C_3=C_2/\ell$ and $a=\dim(M)/4\tau$, this term is $O(N^{a}e^{-C_3\sqrt{N}})\ll N^{{(r-\dim(M))}/{4\tau}}$, so up to changing the constant $K$, we have dealt with the first term.
Finally, we observe that since $A\in C^r$, the projection $A_\lambda$ satisfies \[\|A_\lambda\|_{C^0}\leq \|A\|_{C^r}\|\lambda\|^{-r}.\] Therefore, since $\Lambda$ corresponds to a union of lattices in $\R^{dim(M)}$ whose dimensions sum to $\dim(M)$, we can use the classical bounds on tails of higher dimensional $p$-series to see that for some constant $C_4>0$, \[\sum_{
\|\lambda\|^{2\tau}\geq \sqrt{N}}\|\cL^N A_\lambda\|_{C^0}\leq\sum_{\|\lambda\|>N^{1/4\tau}} \frac{ \|A\|_{C^r}}{\|\lambda\|^r}\leq   C_4N^{-(r-\dim(M))/4\tau}\|A\|_{C^r}.\]
Thus, for the appropriate choice of $K$, $a$, and $b$, the estimate { \eqref{Tame}}
holds.
\end{proof}
\begin{theorem}\label{pencilImpliesMixing}
Let $G/\Gamma$ be an $m$-great nilmanifold. The set of $m$-tuples in $G^m$ which are rapid mixing has full measure with respect to $\mu^m$.
\end{theorem} 
\begin{proof}
Whenever an $m$-tuple acts Diophantinely on every level, Proposition \ref{ThTame} applies. By \eqref{spectralGapToRM}, this is implies rapid mixing for the random walk they generate. By Proposition~\ref{lem;great} the set of such tuples has full measure with respect to $\mu^m$.
\end{proof}

\section{Nice Words}\label{sec;RW}
In this section, we prove Proposition \ref{lem;great}. {Our approach is to  construct words in the semigroup $\langle F\rangle$ that differ by a nonzero element of $G^{(p)}$ for each $p$. In particular, we will find words that differ by a term that can be expressed as a $p$-fold bracket of elements of $\langle F\rangle$. Then, we will observe that a formula for this difference may be given as a polynomial in the exponential coordinates of the $g_i$. As long as the polynomials arising in this manner are linearly independent Theorem \ref{thm;KM} implies that, for a full measure set of walks, the words we construct act Diophantinely on the appropriate level, and thus are nice. This need for linear independence was the direct motivation for Definition \ref{def;great}.} {Once we have linear independence, Theorem \ref{thm;KM} guarantees the existence of a full measure of the set of $m$-tuples that act Diophanitnely and thus rapidly mix.} 

To begin, 
given $p\geq 1$, let $K_p\subset (\Z^m)^{ (p+1)}$ be the set of all indices $\bf{k_0,\ldots,k_p}$ such that 
$[\ldots \bf{[k_0V,k_1V],\ldots],k_p V]}$ can be written of the form $\log(W_1W_2^{-1}) \text{ mod } \g^{(p+1)}$ for two words $W_1$ and $W_2$ of equal length in $\langle F\rangle$. 
{This is the set of $p$-nested commutators that arise as the logarithm of the difference between two elements of our walk at some time. Thus $K_p$ can heuristically be seen as the set of elements of $\g_p$ available to us in trying to get a self-cancellation between an element of $H_\lambda$ pushed forward by two words in our random walk when $\lambda\in \Lambda_p-\{0\}$. 
In constructing elements of $K_p$, the element $k_{ij}$ will count how many times $g_j$ appears in the $i^{th}$ step of the construction of the words differing by $\cH_p(k,\alpha)$. For instance, in the setting of two generators $g_1$ and $g_2$, since \[\log(g_1g_2(g_2g_1)^{-1})=\log([g_1,g_2])=[{\bf k_0V,k_1V}]\] for ${\bf k}_0=(1,0)$ and ${\bf k}_1=(0,1)$, we have that $((1,0),(0,1))\in K_2$.} {In terms of this set, our goal is to find an intersection between $K_p$ and the complement of $\cD(\cH_{m,p})$ i.e. a difference between two words in our walk with linearly independent polynomials. The following lemmas provide a method for finding elements of $K_p$ that will be rich enough {so} that as long as $\cD(\cH_{m,p})\neq \R^{n_p}$, $K_p$ will not be contained in $\cD(\cH_{m,p})$ as we desire.}

We {now describe a method to} build
inductively  words $a_i$ and $b_i$ in $\langle F\rangle$ such that $h_i:=a_ib_i^{-1}\in G^{(i)}$ and to do so with as much flexibility as possible. 
In particular, given any choice of $w_{-1},
\ldots w_p\in \langle F\rangle $ with $|w_{-1}|=|w_0|$, {we can build the sequences 
$a_i$ and $b_i$ in $\langle w_{-1}, w_{0}, \dots, w_p\rangle $}
To begin with, we set $a_0=w_{-1}, b_0=w_0,$
and \begin{align*}a_i=a_{i-1}w_{i}b_{i-1} \quad \text{and} \quad b_{i}=b_{i-1}w_ia_{i-1}.\end{align*} 
Notice that $b_i$ will always be a concatenation of the $w_{-1},\ldots,w_{i}$. 
Since $|w_{-1}|=|w_0|$, this construction will always give us words $a_i$ and $b_i$ satisfying $|a_i|=|b_i|$.
Observe that \[h_1=a_0w_1b_0a_0^{-1}w_1^{-1}b_0^{-1}=[a_0b_0^{-1},b_0w_1],\]
and
\[h_i=a_{i-1}w_{i}b_{i-1}(b_{i-1}w_{i}a_{i-1})^{-1}=a_{i-1}(b_{i-1}^{-1}b_{i-1})w_{i}b_{i-1}a_{i-1}^{-1}w_{i}^{-1}b_{i-1}^{-1}=[a_{i-1}b_{i-1}^{-1},b_{i-1}w_i]\]
Since $a_{i-1}b_{i-1}^{-1}=h_{i-1}$, we can rewrite this as
$h_i=[h_{i-1},b_{i-1}w_i].$
Thus, $h_1$ can be expressed as a single bracket, and $h_i$ can be expressed as a bracket between $h_{i-1}$ and another group element, so inductively $h_i$ can always be expressed as an $i$-fold bracket. In particular, we get the formula \begin{equation}\label{eqn;hpform}  
h_p=[\ldots[[a_0b_0^{-1},b_0w_1],b_1w_2],\ldots, b_{p-1}w_p].\end{equation}

{We now show that this procedure produces a large set of group elements.} 

\begin{lemma}
Any multi-index
$\bk = (\bk_0,\ldots,\bk_p)\in(\Z^{m})^{(p+1)}$, $\bk_i=(k_{i1}, \dots, k_{im})$ satisfying 
$ k_{(i+1)j}>3k_{ij}$, 
is contained in $K_p$.
\end{lemma}

If we can find ${\bf k_0},{\bf k_1},\ldots,{\bf k_p} $ such that 
{$h_p$ in \eqref{eqn;hpform} satisfies}
\[\log h_p=[\ldots \bf{[k_0V,k_1V],\ldots],k_p V]},\] then $({\bf k_0},{\bf k_1},\ldots,{\bf k_p} )\in K_p$.
{Taking} the logarithm of both sides of (\ref{eqn;hpform}) we get \[\log(h_p)=\log( [\ldots[[a_0b_0^{-1},b_0w_1],b_1w_2],\ldots, b_{p-1}w_p]).\] 
The BCH formula allows us to move the logarithm within the brackets modulo some error that belongs to $\g^{(p+1)}$, so
\begin{equation}\label{eqn;logh}\log(h_p)= [\ldots[[
\log(a_0b_0^{-1}),\log(b_0w_1)],\log(b_1w_2)],\ldots, \log(b_{p-1}w_p)]\mod \g^{(p+1)}.\end{equation}
Define a function ${\bf c}:\langle F,F^{-1}\rangle\rightarrow \Z^m$ such that ${\bf c}_i(W)$ is the number of times $g_i$ appears in the word $W$ with inverses counted as negatives. Define ${\bf k_0}={\bf c}(a_0b_0^{-1})$ and ${\bf k_i}={\bf c}(b_{i-1}w_i)$. Applying BCH again, we have that 
$\displaystyle\log (b_{i-1}w_i)={\bf k_i V} \mod \g^{(1)}$. Applying this fact to (\ref{eqn;logh}) and using the fact that the $p$-fold bracket of a $\g^{(1)}$ error is in $\g^{p+1}$, we get
\[\log(h)=[\ldots[\sum_{i=1}^m {\bf k_{0i}}V_i, \sum_{i=1}^m {\bf k_{1i}}V_i],\ldots,\sum_{i=1}^m {\bf k_{pi}}V_i]=[\ldots \bf{[k_0V,k_1V],\ldots],k_p V]}\mod \g^{(p+1)}.\]
Now, we have ${(\bf k_0,\ldots,k_p})\in K_p$ which depends on our initial choice of $w_{-1},\ldots, w_p$. We wish to understand which sequences $\bf{k_0,\ldots,k_p}$ are possible to get as a result of that choice. Note that $\bf{k_0}$ and $\bf{k_1}$ can be taken to be any vector of nonnegative integers simply by making an appropriate choice of $w_{-1}, w_0,$ and $w_1$. Meanwhile, for $i\geq 2$, we have ${\bf k_i}={\bf c}(b_{i-1})+{\bf c}(w_i)$. Since by our choice of $w_i$ we can make ${\bf c}(w_i)$ any nonnegative integer vector we like, this means we can choose any $\bf k_i$, subject to ${\bf k_i}\geq{\bf c}(b_{i-1})$. Since ${\bf c}(b_{i-1})=2{\bf c}(b_{i-2})+{\bf c}(w_{i-1})\leq 3{\bf k}_{i-1}$, if ${\bf k_i}\geq 3{\bf k_{i-1}}$, then ${(\bf k_0,\ldots,k_p})\in K_p$.
\begin{lemma}\label{zDense}
$K_p$ is Zariski dense.
\end{lemma}
\begin{proof}
 By the previous lemma, $\{\k\in (\Z^m)^{p+1}: k_{(i+1)j}>3k_{ij}\}$ is contained in $K_p$. It is clear that this set, 
 is Zariski dense in $\R^{m(p+1)}.$ It follows that $K_p$ is a Zariski dense set.
\end{proof}

\begin{proof}[Proof of Proposition \ref{lem;great}]
 {We claim that there are words $W_1$ and $W_2$ of equal length such that the coordinates of $\pi_p(\log(W_1W_2^{-1}))$ are linearly independent polynomials in $\alpha$. Assume by contradiction that for all words $W_1, W_2$, these polynomials are linearly dependent. That implies that any words that differ by $p$-nested commutators have those commutators' polynomials linearly dependent. Since $K_p$ consists of the set of such realized commutators, this means that
 $K_p\subset \cD(\cH_{m,p}).$ Since $\cD(\cH_{m,p})$ is an algebraic variety and by Lemma \ref{zDense} $K_p$ is Zariski dense, this implies that $\cD(\cH_{m,p})=\R^{n_p}$. Thus, $\cH_{m,p}$ is degenerate, contradicting
the assumption that $\g$ is $m$-great. The claim follows.

Next, by Theorem \ref{thm;KM}, since the coordinate polynomials of $\pi_p(W_1W_2^{-1})$ are linearly independent, 
for almost every value of the $\alpha_{ij}$, $\pi_p(W_1W_2^{-1})\in DC$. That means that for almost every $m$-tuple of $g_i$'s $\pi_p(W_1W_2^{-1})\in DC$.
This implies the proposition. }
\end{proof}

\section{Great Groups}\label{sec;greatGroups}
Now that we have established that $m$-greatness implies rapid mixing of almost all random walks on $m$ elements, we show that this condition is not overly restrictive. To that effect, we establish that many special Lie algebras are 2-great, and that every step-$s$ Lie algebra is $s$-great. This will prove Theorems \ref{allNilmanifoldsTheorem} and \ref{ThRM}. Finally, we present an example of a group that is not 2-great.

{We now briefly discuss our approach for proving that a Lie algebra is $m$-great. Given a Lie algebra $\g$, we will pick some concrete $\bar k_p$ for which we will show that $\cH_{m,p}(\bar k_p,\alpha)$ has linearly independent polynomials i.e. $\bar k_p\notin \cD(\cH_{m,p})$. To show this linear independence, we let $\{P_1,\ldots,P_{n_p}\}$ be the set of $n_p$ coordinate polynomials associated to the $X_i^{(p)}$. If there is a monomial $M_i=\alpha_{i_0j_0}\ldots\alpha_{i_pj_p}$ with nonzero coefficient in $P_i$ but zero coefficient in each $P_j$ when $j\neq i$, we call $M_i$ a \emph{unique monomial} of $P_i$. 
By finding a unique monomial of each of the $P_i$, we can show that the $P_i$ form a linearly independent set. Thus, $\bar k_p\notin \cD(\cH_{2,p})$ and $\cH_{2,p}$ is non-degenerate. From there, we will conclude 2-greatness.}
\subsection{2-great algebras}
Here we prove that all quasi-abelian, triangular and step-3 or lower Lie algebras are 2-great. Throughout this section, we denote $\bar k_p=((1,0),(0,1)^{p})$.  

\begin{proposition}{All quasi-abelian groups are 2-great}\label{qaGood}\end{proposition}
\begin{proof}
    To show this, we must show that $\cH_{2,p}$ is nondegenerate for all $0\leq p\leq s-1$. We recall that we defined a basis $\{X,Y_{i,j}\}_{(i,j)\in J}$ in Section \ref{sec;fil}. We will let 
    $$V_1 =  \alpha_0 X+\sum_j\alpha_{0,j} Y_{0,j} \mod\g^{(1)}\quad \text{and}\quad
 V_2 =  \beta_0 X+\sum_j\beta_{0,j}Y_{0,j} \mod\g^{(1)}$$
 and {let $P_1,\ldots,P_{n_p}$ be the coordinate polynomials
 in $\alpha,\beta$ of $\cH_{2,p}$.
 We will show that if $P_i$ is the polynomial corresponding to the $Y_{i,p}$ coordinate then $\alpha_{0,i}\beta_0^{p}$ is a unique monomial of $P_i$. This is because the coefficient of $\alpha_{0,i}\beta_0^{p}$ in $[\ldots[V_1,V_2],\ldots V_2]$ is exactly $[\ldots[[Y_{0,i},X],X],\ldots,X]=Y_{i,p}$. This means that each $P_i$ has a unique monomial, so they form a linearly independent set of polynomials. 
 Thus $\bar k\notin \cD(\cH_{2,p})$, so $\cH_{2,p}$ is non-degenerate. Since $p$ is arbitrary, any 
 quasi-abelian Lie algebra is 2-great.}
\end{proof}

\begin{proposition}\label{triangle Good}
    For any $s\in\N$, the triangular algebra $\t_s$ is 2-great.
\end{proposition} 
\begin{proof} We recall that we defined a basis $\{E_{ij}\}$ in Section \ref{sec;fil}. Let
\[\DS V_1=\sum_{i=1}^{n-1} \alpha_{0i} E_{i, i+1}
\quad\mathrm{mod}\quad \g^{(1)}\quad \text{ and
} \quad \DS V_2=\sum_{i=1}^{n-1} \beta_{0i} E_{i, i+1}
\quad\mathrm{mod}\quad \g^{(1)}.
\]
{Let $P_1,\ldots,P_{n_p}$ be the coordinate polynomials  in $\alpha,\beta$ of $\cH_{2,p}$.
We will show that if $P_i$ is the polynomial corresponding to the $E_{i,i+p+1}$ coordinate, then $\alpha_{i}\beta_{i+1},\ldots\beta_{i+p}$ is a unique monomial of $P_i$. Indeed the coefficient of $\alpha_{i}\beta_{i+1},\ldots\beta_{i+p}$ in $[\ldots[V_1,V_2],\ldots V_2]$ is exactly $[\ldots[E_{i,i+1},E_{i+1,i+2}],\ldots,E_{i+p,i+p+1}]=E_{i,i+p+1}$. This means that each $P_i$ has a unique monomial, so they form a linearly independent set of polynomials. 
 Thus, $\bar k\notin \cD(\cH_{2,p})$ and so $\cH_{2,p}$ is non-degenerate. Since $p$ is arbitrary, any 
 triangular Lie algebra is 2-great.}
\end{proof}
We now turn out attention to lower step algebras. We treat first the case of step 1 and step 2 algebras as they may be dealt with quite simply.

\begin{proposition}    
    Any abelian or step 2 Lie algebra is 2-great \label{lowisfree}
\end{proposition}

\begin{proof}
We compute $\cH_{2,0}$ and $\cH_{2,1}$ for arbitrary groups. If $\DS V_1 = \sum_{i}  \alpha_i X_i$, $\DS V_2 = \sum_{i} \beta_i X_i$, we have that $\cH_{2,0}((1,0),\alpha)=V_1$ which clearly has linearly independent polynomials for coordinates. Thus $\cH_{2,0}$ is always non-degenerate. {Now, if $\g$ is non-abelian, let $X^{(1)}\in\g^{(1)}$ be a nonzero vector of the form $X^{(1)}=[X_i,X_j]$ for some $i$ and $j$. In $\cH_{2,1}$, the monomial $\alpha_{i}\beta_j$ only appears as a coefficient of 
$ X^{(1)}$. Since vectors of the form $[X_i,X_j]$ form a basis for $\g_1$, this implies that $\cH_{2,1}$ is nondegenerate.} From this we conclude the {proposition}.
\end{proof}
Next, we show that step-3 Lie algebras are also 2-great. This is the best we can do in terms of step as there exist step-4 Lie algebras that are not 2-great as we shall see. {Our proof technique will be similar to the proofs of $2$-greatness that we have already presented in that we begin by fixing some choice of $k$ which we show to be in the complement of $\cD(\cH_{2,2})$. It will differ in that we will show that the existence of a linear dependence between the polynomials of $\cH_{2,2}$ contradicts the fact that $\g$ is step 3 rather than directly showing linear independence by identifying unique monomials of each polynomial. }  
\begin{proposition}\label{step3Good}
Any nilpotent Lie algebra is {of step 3 or less} is
2-great. 
\end{proposition}
\begin{proof}
We set $\DS V_1 = \sum_{i}  \alpha_i X_i$, $\DS V_2 = \sum_{i} \beta_i X_i$ and show that $M_{121}$ consists of linearly independent polynomials.  
By way of contradiction, assume that the polynomials of $M_{121}$ are not independent, i.e.
\[0\equiv \lambda([[V_1,V_2],V_1]) \text{ for some } 0\neq\lambda \in \mathfrak{g}_2^*.\]

Expanding out, we see that this would imply that \[0\equiv \sum_{i_1,i_2,i_3 \in I} \alpha_{i_1}\beta_{i_2}\alpha_{i_3}\lambda([[X_{i_1},X_{i_2}],X_{i_3}])
.\] {From this we deduce that if a given monomial $\alpha_{i_1}\beta_{i_2}\alpha_{i_3}$ has coefficient $W_{i_1i_2i_3}\in\g$ in the expression $[[V_1,V_2],V_1]$, then $W_{i_1i_2i_3}\in Ker(\lambda)$.}
Since terms of the form $\alpha_{i_1}^2\beta_{i_2}$ occur with coefficient $[[X_{i_1},X_{i_2}],X_{i_1}]$, we get
\begin{equation}\label{eqn;Xi123}
   \lambda ([[X_{i_1},X_{i_2}],X_{i_1}]) = 0. 
\end{equation}

{ Next expanding the equation 
 $\lambda([[X_{i_1}+X_{i_3},X_{i_2}],X_{i_1}+X_{i_3}])=0$ as sum of monomials and using \eqref{eqn;Xi123} we get
\begin{equation}
\label{eqn;i1i2i3}
\lambda([[X_{i_1},X_{i_2}],X_{i_3}]+[[X_{i_3},X_{i_2}],X_{i_1}])=0.
\end{equation}
}
Combining  \eqref{eqn;i1i2i3} with
Jacobi identity
$$ [[X_{i_1},X_{i_2}],X_{i_3}] + [[X_{i_2},X_{i_3}],X_{i_1}] + [[X_{i_3},X_{i_1}],X_{i_2}]=0$$
we get
\begin{equation}\label{eqn;123312}
\lambda(2[[X_{i_1},X_{i_2}],X_{i_3}]+[[X_{i_3},X_{i_1}],X_{i_2}])=0.
\end{equation}
Swapping ${i_2}$ and $i_3$ we obtain
\begin{equation}\label{eqn;132213}
\lambda(2[[X_{i_1},X_{i_3}],X_{i_2}]+[[X_{i_2},X_{i_1}],X_{i_3}])=0.
\end{equation}
{ Since the matrix defining  \eqref{eqn;123312}--\eqref{eqn;132213} is non-degenerate, we conclude that}
$$\lambda([[X_{i_1},X_{i_2}],X_{i_3}]) = \lambda([[X_{i_3},X_{i_1}],X_{i_2}]) = 0.$$
Since the equation above holds for arbitrary $i_1,i_2,i_3 \in I$, this implies that $\lambda\!\!=\!\!0$, a contradiction. Thus, $\cH_{2,2}$ is non-degenerate. The fact that $\cH_{2,0}$ and $\cH_{2,1}$ are non-degenerate follows from {Proposition} \ref{lowisfree}, 
so any step-3 (or lower) Lie algebra is 2-great.  
\end{proof}
Combining Propositions \ref{qaGood}, \ref{triangle Good}, and \ref{PrSStepSGreat} with Theorem \ref{pencilImpliesMixing} gives us Theorem \ref{ThRM}.

\subsection{General Lie algebras}
We will now show that every nilpotent Lie algebra is $m$--great for sufficiently large $m$. Our approach will largely mirror that of the proofs of
but taking advantage of the larger value of $m$.
\begin{proposition}
\label{PrSStepSGreat}
    A step--$s$ Lie algebra is $s$--great.
\end{proposition}
\begin{proof}
We prove that $\cH_{s,s-1}$ is non-degenerate.
Let $\DS V_i=\sum_{j=1}^{n_0}\alpha_{ij}X_j \text{ mod } \g^{(1)},$
and consider $\bf{[\ldots[k_0V,k_1V],\ldots],k_{s-1}V]}$. Let $\bar k$ denote the conveniently selected parameters:
 $\mathbf{k_0}=(1,0,\ldots,0),\ \mathbf{k_1}=(0,1,\ldots, 0),\ldots, \mathbf{k_{s-1}}=(0,0,\ldots,1)$. We compute that the coefficient of $\alpha_{0i_0}\cdots\alpha_{s-1i_{s-1}}$ in $\cH_{s,s-1}(\bar k,\alpha)$ is exactly $[\ldots[[X_{i_0},X_{i_1}],\ldots,X_{i_{s-1}}]$, {so $\alpha_{0i_0}\cdots\alpha_{s-1i_{s-1}}$ will be a unique monomial for the corresponding coordinate polynomial of $\cH_{s,s-1}$}. Since each element of the basis for $\g^{(s-1)}$ can be written of the form $[\ldots[[X_{i_0},X_{i_1}],\ldots,X_{i_{s-1}}]$, we get that each coordinate polynomial of $\cH_{s,s-1}(\bar k,\alpha)$ has a {unique} monomial. 
 Thus, the polynomials at $\bar k$ are linearly independent, so $\cH_{s,s-1}$ is non-degenerate. 
 Since every $\g/\g^{(i+1)}$ is step--$ s$ or lower, the same argument will show that
 $\cH_{s,i}$ is non degenerate for all $i<s$. We conclude that $\g$ is $s$--great.   
\end{proof}Combining Proposition \ref{PrSStepSGreat} with Theorem \ref{pencilImpliesMixing} gives us Theorem \ref{allNilmanifoldsTheorem}. 

We also note that $m$-greatness is preserved by taking products and factors.
\begin{proposition}
The property of $m$-greatness is closed under taking direct products or factors in the following sense:\begin{itemize}
    \item If $\g_1$ and $\g_2$ are $m$-great, then $\g_1\times\g_2$ is $m$-great.
    \item If $\g$ is $m$-great and $\h$ is a factor of $\g$, then $\h$ is $m$-great.
\end{itemize}

\end{proposition}
\begin{proof}
    For clarity, in this proof we will let $\cH^\g_{m,p}$ be the map $\cH_{m,p}$ associated to the Lie group $\g.$ Similarly, {let $\alpha_{ij}^\g$ be the Malcev coordinates of the projection into $\g$ of the $V_i$}.\\
    For the first part, observe that $\cH_{m,p}^{\g_1\times\g_2}(k,\alpha^{\g_1\times\g_2})=
     (\cH_{m,p}^{\g_1}(k,\alpha^{\g_1}), \cH_{m,p}^{\g_2}(k,\alpha^{\g_2}))$. Since both of the $\g_i$'s are $m$-great, $\cD(\cH^{\g_i}_{m,p})$ is a positive codimension variety. Thus, there exists a $k_0$ in $\cD(\cH^{\g_1}_{m,p})^c\cap \cD(\cH^{\g_2}_{m,p})^c$. It follows that  $k\in\cD(\cH_{m,p})$. Thus, $\cH_{m,p}^{\g_1\times\g_2}$ is non-degenerate. Since $p$ was arbitrary, we conclude that $\g_1\times\g_2$ is $m$-great.
    
    For the second part, let $\pi:\g\rightarrow\h$ be the factor map. It is straightforward to compute that $\cH_{m,p}^\h=\pi\circ \cH_{m,p}^\g.$ Thus, the coordinates of $\cH^\h_{m,p}$ are the image of the coordinates of $\cH^\g_{m,p}$ under an epimorphism. We conclude that since $\cH^\g_{m,p}$ was non-degenerate, so is 
    $\cH^\h_{m,p}$. Therefore, $\h$ is $m$-great. 
\end{proof}

\subsection{Counter Example}\label{sec;conc}
{ Although $m$-greatness is a great property}, it has some weaknesses. In particular, although it is expected that on any nilmanifold almost any random walk generated by two translations will be rapid mixing, not all Lie algebras are $2$--great. 
We provide an explicit example of such a Lie algebra. 
\begin{example}\label{2-bad}
Let $\mathfrak{g}$ be a step-4, 15-dimensional Lie algebra with a basis\break $\{X_1, X_2, X_3,Y_1,Y_2,Y_3,Z_1,\ldots ,Z_8,W\}$
and the following commutation relations:
\begin{gather*}
\begin{split}
\label{ex;comm}
[X_1, X_2] = Y_1, \quad [X_1, X_3] = Y_2, \quad [X_2, X_3] = Y_3,\\
[Y_1, X_1] = Z_1, \quad [Y_1, X_2] = Z_2, \quad [Y_1, X_3] = Z_3,\\
[Y_2, X_1] = Z_4, \quad [Y_2, X_2] = Z_5, \quad [Y_2, X_3] = Z_6,\\
[Y_3, X_1] = Z_5-Z_3, \quad [Y_3, X_2] = Z_7, \quad [Y_3, X_3] = Z_8,\\
[Z_1,X_3] = 3W, \quad [Z_2,X_3] = -3W, \quad [Z_3,X_1] = -W,\\
[Z_3,X_2] = W, \quad [Z_3,X_3] = -2W, \quad [Z_4,X_2]=-3W,\\
 [Z_5,X_1] = W,\quad [Z_5,X_2]=2W, \quad [Z_5,X_3]=-W,\\
 [Z_6,X_2] = 3W, \quad [Z_7,X_1] = -3W, \quad  [Z_8,X_1] = -3W,\\
 [Y_1,Y_2] = 4W, \quad [Y_1,Y_3] = -4W, \quad  [Y_2,Y_3] = -4W,
\end{split}
\end{gather*}
while other commutation relations are zero. That these relations actually define a Lie algebra can be verified computationally, but we do not reproduce the calculations here.
 \end{example}

\begin{proposition}
\label{PrTooBad}
$\g$ is 2-bad.
\end{proposition}
 \begin{proof}  
We compute that $\cH_{2,3}\equiv 0$ even though $\g$ is step 4.
By symmetry, it suffices to check that 
\begin{equation}\label{commv1v2}
M_{1211}=[[V_1,V_2],V_1],V_1]=0, \quad M_{1212}=[[V_1,V_2],V_1],V_2] =0.
\end{equation} 
This may be verified via a computer program, and we present the detailed calculation in Appendix \ref{Liecal}.
 \end{proof}
 This implies that for any nilmanifold which is a quotient of the Lie group associated with $\g$, our techniques fail to show rapid mixing of any walk generated by 2 elements. The question of whether such walks are rapid mixing remains open.

\appendix
\section{Lemma for application of Diophantine estimate}
The following inequality is useful in estimating the norm of $\cL$.
\begin{lemma}\label{rem:new}
There exists a $C>0$ such that for all $\theta\in[-1/2,1/2]$, 
\begin{equation}
\label{NonColl}
|1+e^{2\pi i\theta}|\leq 2- C\theta^2.
\end{equation}
\end{lemma}
\begin{proof}
Rearranging, we rewrite \eqref{NonColl} 
as \[\displaystyle C\leq\inf_{\theta\in[-1/2,1/2]}\frac{(2-|1+e^{2\pi i\theta }|)}{\theta^2}.\]
Applying law of cosines and the double angle identity, the inequality becomes \[C\leq \inf_{\theta\in[-1/2,1/2]} \frac{2-2\cos(\pi\theta)}{\theta^2}.\] 
The singularity of the RHS is removable since it can be rewritten as 
$\DS g(\theta)\!\!:=\!2\!\!\sum_{n=0}^\infty \frac{(\pi\theta)^{2n}}{(2n+2)!}$.
Since $g$ is positive and continuous on $[-1/2, 1/2]$ its achieves its minimum value $C.$ 
\end{proof}
{We comment that, in fact, the Lemma holds with $C=8$ but we will not use this in our arguments.}
{\begin{lemma}\label{rem:newer}
If $v\in DC(\gamma,\tau)$ and $\lambda$ is a nonzero linear functional with integer coefficients, then there exists $C'>0$ depending only on $\gamma$ such that \[|1+e^{2\pi i\lambda(v)}|\leq 2-C'|\lambda|^{2\tau}.\]
\end{lemma}
\begin{proof}
Based on the Diophantine property of $v$, $\lambda(v)$ is congruent mod 1 to a number $\theta$ in $[-1/2,1/2]$ satisfying $|\theta|>\gamma |\lambda|^{-\tau}$. It follows by Lemma \ref{rem:new} that \[|1+e^{2\pi i\lambda(v)}|=|1+e^{2\pi i \theta}|\leq 2-C\gamma^2|\lambda|^{-2\tau}.\] Setting $C'=C\gamma^2$, we have completed the proof.
\end{proof}}
\section{Central Limit Theorem}
\label{AppCLT}
We say that $x_n$ satisfies CLT if there is $r>0$ such that for any function $A \in C^r(M)$ with zero mean, $\DS \frac{1}{\sqrt{N}}\sum_{n=0}^{N-1}A(x_n)$ converges in distribution to a Gaussian random variable with zero expectations.

\begin{corollary}\label{cor;CLT}
If $\cL$ satisfies \eqref{Tame} then $x_n$ satisfies the CLT.
\end{corollary}

While this result is standard 
we include the proof for completeness.

\begin{proof}
Let $\DS B=(1-\cL)^{-1} A=\sum_{n=0}^\infty \cL^n A.$ By \eqref{Tame} this series converges
in $C^0$ for $r$ large enough. Thus $B(x_n)=A(x_n)+\EXP(B(x_{n+1}|\cF_n)$
where $\cF_n$ is the $\sigma$-algebra generated by $(x_0, \dots, x_n).$
Summing up, we obtain 
$\DS \sum_{n=0}^{N-1} A(x_n)= \sum_{n=1}^N \Delta_n+B(x_0)-B(x_N) $
where $\Delta_n\!\!=\!\!B(x_n)\!-\!\EXP(B(x_n)|\cF_{n-1})$ is a martingale difference
sequence. Now by the CLT for martingales (see e.g. \cite[\S 8.2]{Du}), to prove the CLT, 
it is sufficient to show that the limit
$\DS \sigma^2:=\lim_{N\to\infty}  \frac{1}{N} \sum_{n=0}^{N-1} q_n $
exists (in probability) where $q_n=\EXP(\Delta_n^2|\cF_{n-1}).$
 Note that in our case $q_n=Q(x_{n-1})$ for a continuous function $Q$ so the existence
 of the limit follows from the ergodicity of our Markov chain.
\end{proof}

\section{Lie algebra calculation}\label{Liecal}
{ Here we provide the computations relevant to Proposition \ref{PrTooBad}.} 
It suffices to check that $[[[V_1,V_2],V_1],V_1]=[[[V_1,V_2],V_1],V_2]=0$ since all other triply nested brackets are either necessarily equal to one of these (up to a sign) or forced to be 0. 
Assume $\DS V_1=\sum_{i=1}^3  \alpha_i X_i,\;\; V_2= \sum_{i=1}^3  \beta_i X_i$ and we compute the following.
\medskip

\begin{align*}
    [[V_1,V_2],V_1] = & (\alpha_1^2\beta_2 -\alpha_1\alpha_2\beta_1)Z_1 + (\alpha_1\alpha_2\beta_2-\alpha_2^2\beta_1)Z_2 + (\alpha_1\alpha_3\beta_2 - \alpha_2\alpha_3\beta_1)Z_3+\\
    &(\alpha_1^2\beta_3-\alpha_1\alpha_3\beta_1)Z_4+(\alpha_1\alpha_2\beta_3-\alpha_2\alpha_3\beta_1)Z_5+(\alpha_1\alpha_3\beta_3-\alpha_3^2\beta_1)Z_6+\\
    &(\alpha_1\alpha_2\beta_3-\alpha_1\alpha_3\beta_2)(Z_5-Z_3)+(\alpha_2^2\beta_3-\alpha_2\alpha_3\beta_2)Z_7+(\alpha_2\alpha_3\beta_3-\alpha_3^2\beta_2)Z_8.
\end{align*}
    
    \begin{align*}
    [[[V_1,&V_2],V_1],V_2]  = \Big((\alpha_1^2\beta_2\beta_3-\alpha_1\alpha_2\beta_1\beta_3) -
    3(\alpha_1\alpha_2\beta_2\beta_3 - \alpha_2^2\beta_1\beta_3)-
    (\alpha_1\alpha_3\beta_1\beta_2-\alpha_2\alpha_3\beta_1^2)\\
    &+(\alpha_1\alpha_3\beta_2^2-\alpha_2\alpha_3\beta_1\beta_2) -
     3(\alpha_1\alpha_3\beta_2\beta_3 - \alpha_2\alpha_3\beta_1\beta_3)-3(\alpha_1^2\beta_2\beta_3-\alpha_1\alpha_3\beta_1\beta_2)\\
    &+(\alpha_1\alpha_2\beta_1\beta_3 - \alpha_2\alpha_3\beta_1^2) + 2(\alpha_1\alpha_2\beta_2\beta_3 - \alpha_2\alpha_3\beta_1\beta_2) - 
    (\alpha_1\alpha_2\beta_3^2 - \alpha_2\alpha_3\beta_1\beta_3)\\
    &+ 3(\alpha_1\alpha_3\beta_2\beta_3 - \alpha_3^2\beta_1\beta_2) - 3(\alpha_2^2\beta_1\beta_3 - \alpha_2\alpha_3\beta_1\beta_2)-
    3(\alpha_2\alpha_3\beta_1\beta_3 - \alpha_3^2\beta_1\beta_2) \\
    &+ 2(\alpha_1\alpha_2\beta_1\beta_3 - \alpha_1\alpha_3\beta_1\beta_2) +
    (\alpha_1\alpha_2\beta_2\beta_3 - \alpha_1\alpha_3\beta_2^2) + (\alpha_1\alpha_2\beta_3^2 - \alpha_1\alpha_3\beta_2\beta_3)\Big)W
    =0. 
    \end{align*}
 \begin{align*}
     [[[V_1,&V_2],V_1],V_1] = \Big(3(\alpha_1^2\alpha_3\beta_2-\alpha_1\alpha_2\alpha_3\beta_1) -
    3(\alpha_1\alpha_2\alpha_3\beta_2 - \alpha_2^2\alpha_3\beta_1)-
    (\alpha_1^2\alpha_3\beta_2-\alpha_1\alpha_2\alpha_3\beta_1)\\
    &+
    (\alpha_1\alpha_2\alpha_3\beta_2-\alpha_2^2\alpha_3\beta_1) -
     3(\alpha_1\alpha_3^2\beta_2 - \alpha_2\alpha_3^2\beta_1)-3(\alpha_1^2\alpha_2\beta_3-\alpha_1\alpha_2\alpha_3\beta_1)\\
    &+(\alpha_1^2\alpha_2\beta_3 - \alpha_1\alpha_2\alpha_3\beta_1) + 2(\alpha_1\alpha_2^2\beta_3 - \alpha_2^2\alpha_3\beta_1) - 
    (\alpha_1\alpha_2\alpha_3\beta_3 - \alpha_2\alpha_3^2\beta_1)\\
    &+ 3(\alpha_1\alpha_2\alpha_3\beta_3 - \alpha_2\alpha_3^2\beta_1) - 3(\alpha_1\alpha_2^2\beta_3 - \alpha_1\alpha_2\alpha_3\beta_2)-
    3(\alpha_1\alpha_2\alpha_3\beta_3 
- \alpha_1\alpha_3^2\beta_2)\\ 
&+ 2(\alpha_1^2\alpha_2\beta_3 - \alpha_1^2\alpha_3\beta_2) 
    +(\alpha_1\alpha_2^2\beta_3 - \alpha_1\alpha_2\alpha_3\beta_2) + (\alpha_1\alpha_2\alpha_3\beta_3 - \alpha_1\alpha_3^2\beta_2)\Big)W=0. 
    \end{align*}


 {     In other words $M_{1211}=M_{1212}=0$, so 
  $\cH_{2,3}\equiv 0$ even though $\g$ is step-4.}

\section*{Statements and Declarations}

   The authors have no relevant financial or non-financial interests to disclose.
   Data sharing is not applicable to this article as no datasets were generated or analysed during the current study.


\end{document}